\documentclass[smallextended,numbook,runningheads]{svjour3}  
\smartqed  

\usepackage{graphicx}
\usepackage{amsmath}
\usepackage{mathptmx}      
\usepackage{amsfonts}

\journalname{}

\usepackage{amssymb,amscd}
\usepackage{graphicx}
\usepackage{epstopdf} 
\usepackage{mathptmx}      
\usepackage{color}

\usepackage{mdframed}
\usepackage{systeme,mathtools}
\usepackage{tabstackengine}
\usepackage[]{algorithm2e}

\stackMath

\def\L{\mathcal{L}}

\usepackage{graphicx}
\usepackage{amsmath}
\usepackage{mathptmx}      
\usepackage{amsfonts}
\usepackage{booktabs}

\begin{document}
	
\title{New Stability Results for Explicit Runge-Kutta Methods}

\titlerunning{Stability of explicit Runge-Kutta methods}        

\author{Rachid Ait-Haddou}

\authorrunning{R. Ait-Haddou} 

\institute{Rachid Ait-Haddou  \at
          Cybermedia Center 6F, Osaka University, 1-32 Machikaneyama, Toyonaka 560-0043, Osaka, Japan \\
	      \email{rachid.aithaddou70@gmail.com  \\
	      }        
}

\date{Received: date / Accepted: date}

\maketitle

\begin{abstract}
The theory of polar forms of polynomials is used to provide for sharp bounds on the radius of 
the largest possible disc (absolute stability radius), and on the length of the largest possible 
real interval (parabolic stability radius), to be inscribed in the stability region of an explicit 
Runge-Kutta method. The bounds on the absolute stability radius are derived as a consequence of 
Walsh's coincidence theorem, while the bounds on the parabolic stability radius are achieved by using 
Lubinsky-Ziegler's inequality on the coefficients of polynomials expressed in the Bernstein bases
and by appealing to a generalized variation diminishing property of B\'ezier curves.
We also derive inequalities between the absolute stability radii of methods with different 
orders and number of stages.

\end{abstract}      
\section{Introduction}
\label{intro}
Runge-Kutta methods are the most widely used numerical schemes for solving initial 
value ordinary differential equations of the type 
\begin{equation}
\label{DiffEquation}
\frac{dy}{dx} = f(x,y), \quad y(x_{0}) = y_{0},
\end{equation}
with $y: \mathbb{R} \longrightarrow \mathbb{R}^s; \; s \geq 1$ and $f(x,y)$ has value in $\mathbb{R}^s$. 
When applied to the Dahlquist scalar test $y' = \lambda y$, an explicit single-step, 
$m$-stages Runge-Kutta method yields a numerical scheme of the form 
\begin{equation*}
y_{k+1} = P_{m}(\lambda h) y_{k},
\end{equation*}
where $h$ is the step-size and $P_{m}$ is a polynomial of degree at most $m$ called
the {\it{stability polynomial}} of the explicit Runge-Kutta method. 
The stability polynomial $P_{m}$ is compatible with a Runge-Kutta method of order $n$ 
if it is of the form 
\begin{equation}
\label{Expo}
P_{m}(x) = \sum_{j=0}^{n} \frac{x^j}{j!} + \sum_{j=n+1}^{m} \alpha_{j} x^j.
\end{equation}   
The {\it{stability region}} $S_{P_{m}}$ of a Runge-Kutta method with stability polynomial $P_{m}$  
is defined by
\begin{equation*}
S_{P_{m}} = \{z \in \mathbb{C} \; | \;  |P_{m}(z)| \leq 1 \}.
\end{equation*}
An explicit Runge-Kutta method for solving (\ref{DiffEquation}) is said to be 
{\it{linearly stable}} if each value $\lambda h$, with $\lambda$ an eigenvalue 
of the Jacobian of $f$, belongs to the stability region $S_{P_{m}}$ of the method. 
This concept of linear stability is of great relevance to the numerical integration
of (\ref{DiffEquation}) since, locally at least, one can view (\ref{DiffEquation}) as a small 
perturbation of a linear system.
 
The design of efficient explicit Runge-Kutta methods for solving (\ref{DiffEquation}) 
should aim at step-sizes as large as possible without destroying the linear stability 
and the order of accuracy constraints. Such  strategies depend solely on the spectrum 
of the Jacobian of $f$. For spectra in general position, such as in stiff 
differential equations, one should pursue methods with stability polynomials 
of the from (\ref{Expo}) and whose stability region contains the largest 
possible disc \cite{nevanlinna,vich}. The radius of such a disc is called the 
{\it{absolute stability radius}}. However, when dealing with the semi-discretization 
of parabolic partial differential equations, the local spectrum of the corresponding 
differential equations consists, in general, of negative real numbers. In such situations, 
one should aim at methods with stability polynomials of the from (\ref{Expo}) and whose stability 
region contains the largest possible negative real interval \cite{abdulle1,abdulle2,boga,lawson,medo,renault,vander10}.
The length of this interval is usually called the {\it{parabolic stability radius}}.
Similarly, it is well known that Jacobian with spectrum lying on the imaginary axis appears 
in the semi-discretization of hyperbolic partial differential equations, and thus polynomials of the form 
(\ref{Expo}) whose stability region contains the largest possible interval in the imaginary axis 
are of great relevance in such situations \cite{kinn1,kinn2,mead}. The length of this interval is 
called the {\it{hyperbolic stability radius}}. 

Many numerical methods for computing these stability radii and their corresponding 
optimal polynomials rely on some form of dichotomy \cite{ketcheson,vander,riha}. 
Therefore, establishing bounds for these stability radii is of great importance for 
efficient initialization of any dichotomy algorithm.

In this paper, we give sharp bounds on the absolute and parabolic stability radii 
of methods with given number of stages and order of accuracy. 
Our main tool is the theory of polar forms of polynomials \cite{ramshaw}.
The bounds on the absolute stability radius are derived as a consequence of Walsh's coincidence theorem.
Moreover, we show that the theory of polar forms leads to interesting inequalities between the absolute 
stability radii of methods with different orders and number of stages. 
Sharp upper bounds on the parabolic stability radius are achieved using the Lubinsky-Ziegler's inequality 
on the coefficients of polynomials expressed in the Bernstein bases. The lower bounds are derived as 
a consequence of a generalized variation diminishing property of B\'ezier curves. Moreover,  we show 
that a generalization of Lubinsky-Ziegler's inequality leads to an upper bound on the parabolic stability 
radius of optimal methods with damping. 

The methodology presented in this paper is not specific to the family of polynomials of the form (\ref{Expo}) 
and can be applied to study optimal polynomials with different constraints on the coefficients. 
For instance, the linear stability theory of splitting methods \cite{McLachlan1,McLachlan2} leads 
to the study of optimal polynomials of the form
\begin{equation}
\label{Cos}
P_{m}(x) = \sum_{j=0}^{2n} (-1)^j \frac{x^{2j}}{(2j)!} + \sum_{k=n+1}^{m} \alpha_{j} x^{2j},
\end{equation}  
and the methods presented in this paper can be easily adapted to give bounds on the radius of 
optimal polynomials of the form (\ref{Cos}).   
   
The paper is organized as follows: In section \ref{Sec2} we gather several technical results 
that are fundamental for the rest of the paper. In particular, we study the zeros of a family 
of polynomials related to the generalized Bessel polynomials, we introduce the notions of polar 
forms and B\'ezier curves and we establish a connection between the polar forms of polynomials 
of the form (\ref{Expo}) and generalized Laguerre polynomials with negative parameters.
In section \ref{Sec3}, we give a sharp upper bound on the absolute stability radius using Walsh's 
coincidence theorem. Moreover, we present simple and self-contained new proofs for 
the explicit expression of the optimal polynomials of first and second order. The techniques 
used in these proofs are generalized to give inequality between the absolute stability radii of
method with different orders and number of stages. In Section \ref{Sec4}, we give sharp bounds on 
the parabolic stability radius. The upper bound is achieved using Lubinsky-Ziegler's inequality 
on the coefficients of polynomials expressed in the Bernstein bases, while the lower 
bound is established using a refined variation diminishing property of B\'ezier curves. 
Moreover, we generalize Lubinsky-Ziegler's inequality to provide for an upper bound on the 
parabolic stability radius with damping. We conclude in Section \ref{Sec5} with 
some remarks and future work.                
           
\section{Preliminary results}
\label{Sec2}
In this section we gather several technical results that we shall use throughout the paper. 
In particular we introduce the notion of polar form and show the relation between the polar 
form of truncated exponential sums and generalized Laguerre polynomials with negative 
parameter.
\subsection{Zeros of a family of polynomials related to the generalized Bessel polynomials}
Let $n$ be a positive integer and $\alpha$ be a positive real number. 
Define the polynomials $G_{n}(\alpha,.)$ by
\begin{equation}
\label{Gfunction}
G_{n}(\alpha,y) =  \sum_{i=0}^{n} \binom{n}{i} (\alpha)_{i} \; y^{n-i},
\end{equation} 
where $(\alpha)_{i} = 1$ if $i=0$ and  $(\alpha)_{i} = \alpha(\alpha+1)(\alpha+2)\ldots(\alpha+i-1)$
for $i\geq 1$. It is observed in \cite{owrenthesis} that the polynomials $G_{n}(\alpha,.)$ 
can also be expressed as 
\begin{equation}
\label{BesselR}
G_{n}(\alpha,y) = 2^{n} \Theta_{n}(\frac{y}{2},\alpha-n+1),
\end{equation}
where $\Theta_{n}(z,a)$ are the generalized reverse Bessel polynomials \cite{grosswald}. 
Based on (\ref{BesselR}) and Theorem 1, p. 80 of \cite{grosswald}, one deduces that
for $n$ even, $G_{n}(\alpha,y) >0$ for any $y \in \mathbb{R}$, while for odd $n$, 
the polynomial $G_{n}(\alpha,.)$ has precisely one real zero that is simple and negative.

Let $\beta$ be a real number such that
\begin{equation}
\label{alphabetacondition}
\beta \geq (\alpha)_{n}. 
\end{equation}
Denote by $R_{n}(\alpha,\beta,.)$ the polynomials 
\begin{equation*}
R_{n}(\alpha,\beta,y) = \beta - (-1)^n G_{n}(\alpha,y).
\end{equation*}
Using the above mentioned properties of the polynomials $G_{n}(\alpha,.)$ and the easily 
verified relations 
\begin{equation*}
\frac{\partial R_{n}}{\partial y}(\alpha,\beta,y) = (-1)^{n+1} n G_{n-1}(\alpha,y) 
\quad \textnormal{and} \quad
R_{n}(\alpha,\beta,0) = \beta - (-1)^n (\alpha)_{n},
\end{equation*}
one can readily obtain the following.
\begin{proposition}
\label{rootsb}
\enumerate{
\item For any positive odd integer $n$ and for any fixed positive real numbers $\alpha, \beta$
such that $\beta \geq (\alpha)_{n}$, the polynomial $R_{n}(\alpha,\beta,.)$ has precisely 
one real root, $y_{0}(\alpha,\beta,n)$, which is simple and negative. 
\item For any positive even integer $n$ and for any fixed positive real numbers $\alpha, \beta$
such that $\beta >(\alpha)_{n}$, the polynomial $R_{n}(\alpha,\beta,.)$ has precisely 
one real root, $y_{0}(\alpha,\beta,n)$, which is simple and negative. 
\item When $n$ is even and $\beta = (\alpha)_{n}$, the polynomial $R_{n}(\alpha,\beta,.)$ 
has $y=0$ as a root and precisely one real root, $y_{0}(\alpha,\beta,n)$, which is simple and negative.  
}  
\end{proposition}
\vskip 0.2 cm

We are interested in providing lower bounds to the root $y_{0}(\alpha,\beta,n)$ defined 
in Proposition (\ref{rootsb}). For $n$ odd, denote by $\eta_{0}(\alpha,n)$ the unique negative 
root of $G_{n}(\alpha,.)$. Using the condition (\ref{alphabetacondition}), one can easily show that 
for $n$ odd, the equation $G_{n}(\alpha,y) = \beta$ possesses a unique non-negative zero 
that we shall denote $\xi_{0}(\alpha,\beta,n)$. It is shown in \cite{owrenthesis} that 
\begin{equation*}
-(\alpha+n-1) \leq \eta_{0}(\alpha,n) \leq -\alpha.
\end{equation*}
Moreover, a simple upper bound to $\xi_{0}(\alpha,\beta,n)$ is given by the Cauchy 
bound \cite{oberkoff}
\begin{equation*}
\xi_{0}(\alpha,\beta,n) \leq \left(\beta-(\alpha)_{n}\right)^{1/n}.
\end{equation*}
It is also shown in \cite{owrenthesis} that for odd $n$; for any  $\epsilon >0$ we have  
\begin{equation}
\label{Owren1}
G_{n}\left(\alpha,\eta_{0}(\alpha,n) + \epsilon\right) \leq
|G_{n}\left(\alpha,\eta_{0}(\alpha,n) - \epsilon\right)|.
\end{equation}
Using this inequality, we now show that, for $n$ odd, the unique negative 
zero $y_{0}(\alpha,\beta,n)$ of the polynomial $R_{n}(\alpha,\beta,.)$ 
satisfies 
\begin{equation*}
y_{0}(\alpha,\beta,n) \geq 2\eta_{0}(\alpha,n) - \xi_{0}(\alpha,\beta,n) 
\geq -2(\alpha+n-1) - (\beta-(\alpha)_{n})^{1/n}.
\end{equation*}
We proceed by contradiction by assuming that 
$y_{0}(\alpha,\beta,n) < 2\eta_{0}(\alpha,n) - \xi_{0}(\alpha,n)$. Taking 
$\epsilon = \eta_{0}(\alpha,n) - y_{0}(\alpha,\beta,n)$ in (\ref{Owren1}) and invoking  
the strict monotonicity of $G_{n}(\alpha,.)$ when $n$ is odd, we obtain 
\begin{equation*}
|G_{n}(\alpha,y_{0}(\alpha,\beta,n))| \geq G_{n}(\alpha,2\eta_{0}(\alpha,n)- y_{0}(\alpha,\beta,n)) >     
G_{n}(\alpha,\xi_{0}(\alpha,\beta,n)) = \beta.
\end{equation*}
Thus, we obtain the contradictory claim that $R_{n}(\alpha,\beta,y_{0}(\alpha,\beta,n)) < 0$.  
Similarly, for $n$ even and using the fact that any $\epsilon >0$ (see \cite{owrenthesis}) 
\begin{equation*}
G_{n}(\alpha,\eta_{0}(\alpha,n-1) + \epsilon) \leq 
|G_{n}(\alpha,\eta_{0}(\alpha,n-1) - \epsilon)|,
\end{equation*}
we can conclude that 
\begin{equation*}
y_{0}(\alpha,\beta,n) \geq 2 \eta_{0}(\alpha,n-1) - \xi_{0}(\alpha,n) \geq  
-2(\alpha+n-2) - (\beta-(\alpha)_{n})^{1/n}.
\end{equation*}
Summarizing. 
\begin{theorem}
\label{Th:Rfunction}
For any positive integer $n$ and for any positive real numbers $\alpha, \beta$
such that $\beta \geq (\alpha)_{n}$, the only negative root, $y_{0}(\alpha,\beta,n)$, of 
the polynomial $R_{n}(\alpha,\beta,.)$ satisfies 
\begin{equation*}
y_{0}(\alpha,\beta,n) \geq  -2 \left(\alpha+n -\frac{3}{2} - 
\frac{(-1)^n}{2}\right) - \left(\beta-(\alpha)_{n}\right)^{1/n}.
\end{equation*}
\end{theorem}
\subsection{Generalized Laguerre polynomials with negative parameters}
The classical Laguerre polynomials $\L^{(\gamma)}_{n}$ are orthogonal on the 
interval $[0,\infty)$ with respect to the weight $x^{\gamma} e^{−x}$, that is
\begin{equation}
\label{LaguerreIntegral}
\int_{0}^{\infty} \L_{n}^{(\gamma)}(x) \L_{m}^{(\gamma)}(x)  x^{\gamma} e^{-x} = 0, 
\quad \textnormal{if} \quad n \not= m.  
\end{equation}
The integral in (\ref{LaguerreIntegral}) converges only if $\gamma >-1$. 
Explicit expressions of Laguerre polynomials are given by 
\begin{equation}
\label{ExplicitLaguerre}
\L_{n}^{(\gamma)}(x) = \sum_{\ell=0}^{n} \binom{n+\gamma}{n-\ell} \frac{(-x)^\ell}{\ell!}.
\end{equation}
Here, as usual, the binomial coefficient $\binom{t}{\ell}$ is defined by
\begin{equation*}
\binom{t}{\ell} = \frac{t(t-1)\ldots (t-\ell+1)}{\ell!} 
\quad \textnormal{for} \quad \ell >0 \quad 
\textnormal{and} \quad \binom{t}{0} =1.
\end{equation*}
Formula (\ref{ExplicitLaguerre}) makes sense even for negative parameters $\gamma$ and
will be taken as the definition of generalized Laguerre polynomials with negative parameters.
Direct computation shows a simple relation between the generalized Laguerre polynomials 
and the polynomials $G_{n}(\alpha,.)$ defined in (\ref{Gfunction}), i.e.;
\begin{equation}
\label{GpLaguerre}
G_{n}(\alpha,x) = (-1)^n n! \; \L_{n}^{(-\alpha-n)}(x).
\end{equation}
With the aid of Proposition \ref{rootsb} and Theorem \ref{Th:Rfunction}, we show the following.
\begin{corollary}
\label{MainCorollary}
Let $m\geq n \geq 1$ be given integers and $b$ be a real number such that 
$b \geq \binom{m}{n}$. Then
\begin{equation}
\label{claim1}
\L_{n}^{(-m-1)}(x) + b > 0 \quad \textnormal{for any} \quad x \in ]-\infty,0[.
\end{equation}
and the polynomial equation $\L_{n}^{(-m-1)}(x)-b = 0$ has a unique solution $\mu_{m,n}$ 
in the interval  $]-\infty,0[$ and it satisfies 
\begin{equation*}
\mu_{m,n} \geq  -(b n!- (m-n+1)_{n})^{1/n} -2 m+(1 + (-1)^n).  
\end{equation*}
\end{corollary}
\begin{proof}
From (\ref{GpLaguerre}), we have 
\begin{equation*}
n! \left(\L_{n}^{(-m-1)}(x) + b\right) = (-1)^n G_{n}(m-n+1,-x) + n! b.
\end{equation*}
For $n$  even, we have $G_{n}(m-n+1,x) > 0$ for any $x \in \mathbb{R}$ 
and the claim (\ref{claim1}) follows. For $n$ odd, the function 
$G_{n}(m-n+1,.)$ is strictly increasing over $\mathbb{R}$ and thus 
for $x < 0$ we have
\begin{equation*}
(-1)^n G_{n}(m-n+1,x) + b > -G_{n}(m-n+1,0) + n! b = -\binom{m}{n} n! + n! b \geq 0.
\end{equation*}   
This concludes the proof of (\ref{claim1}).
To prove the second claim of the corollary, we observe that the zeros of the polynomial 
$\L_{n}^{(-m-1)}(x)-b$ coincide with the zeros of the polynomial 
$R_{n}(m-n+1, n! b,x) = n!b-(-1)^n G_{n}(m-n+1,-x)$. 
Since $b \geq \binom{m}{n}$, the proof then follows from 
Theorem \ref{Th:Rfunction}.\qed
\end{proof}
\subsection{Polar forms of polynomials}
\label{subSec3}
Polar forms (or blossoms) of polynomials \cite{ramshaw} are crucial tools in various 
mathematical areas \cite{aithaddou1,aithaddou2,aithaddou3,aithaddou4,aithaddou6}.
They will prove essential in this work. 

Let $P$ be a complex polynomial of degree at most $n \geq 1$, then for any complex numbers  
$a, b$ with ($a \neq b$), we can express the polynomial $P$ in the Bernstein basis over $(a,b)$ as  
\begin{equation*}
P(z) = \sum^n_{i = 0}p_i B_i^n (z),
\end{equation*}
where the Bernstein polynomials $B^n_i(z)$, $i=0,...,n$ are given by 
\begin{equation*}
B_{i}^{n}(z) = \binom{n}{i}\left(\frac{b-z}{b-a}\right)^{n-i} \left(\frac{z-a}{b-a}\right)^{i}.
\end{equation*}    
The complex numbers $p_i, i=0,\ldots,n$ are called the control points of $P$ over $(a,b)$ and 
the polygon $(p_0,p_1,...,p_n)$ is called the control polygon of $P$ over $(a,b)$.

\begin{definition}
Let $P$ be a polynomial of degree at most $n \geq 1$. There exists a unique
multi-affine, symmetric function in $n$ variables $p$: $\mathbb{C}^n$
$\longrightarrow \mathbb{C}$ such that for any $z$ in $\mathbb{C}$
we have $p(z,z,\ldots,z) = P(z)$. The function $p$ is called the polar form or the blossom
of the polynomial $P$.
\end{definition}
The control polygon $(q_0,q_1,...,q_n)$ of a polynomial $P$ over $(c,d)$ with 
($c \neq d$) can be computed using its polar form $p$ as 
\begin{equation*}
q_i = p\left(c^{[n-i]}, d^{[i]}\right), \quad  i = 0, 1,\ldots,n.
\end{equation*}
Here, and throughout the paper, the notation $z^{[k]}$ indicates that the complex number $z$ is
to be repeated $k$ times. If the polynomial $P$ is expressed in the monomial basis 
as $P(z) = \sum_{k=0}^n a_{k} z^k$ then its polar form is given by 
\begin{equation}
\label{ExplicitBlossom}
p(u_{1},u_{2},\ldots,u_{n}) = \sum_{k=0}^n a_{k} 
\frac{\sigma_{k}(u_{1},u_{2},\ldots,u_{n})}{\binom{n}{k}},
\end{equation}
where $\sigma_{k}$ refers to the $k-$th elementary symmetric functions in the variables
$u_1,\ldots,u_{n}$, i.e;
\begin{equation*} 
\sigma_{k}(u_1,u_2,\ldots,u_n) =\sum_{1 \leq j_{1} <\ldots<j_{k} \leq n} 
u_{j_{1}} u_{j_{2}} \ldots u_{j_{k}}.
\end{equation*}
From now on, we denote by $\Pi_{m,n}$ the class of polynomials of the form 
\begin{equation*}
P(z) = \sum_{k=0}^{m} \alpha_{k} \frac{z^k}{k!}, 
\quad \textnormal{where} \quad \alpha_{0}=\alpha_{1}=\ldots=\alpha_{n}=1, \quad 0 \leq n \leq m.
\end{equation*}  

The following relation between the polar form of polynomials in $\Pi_{m,n}$ and generalized 
Laguerre polynomials with negative parameters is essential in this work and is easily 
proven using (\ref{ExplicitBlossom}).    
\begin{proposition}
\label{PolarLaguerre}
Let $P$ be a polynomial in $\Pi_{m,n}$. For $k=0,1,\ldots, n,$ we have 
\begin{equation}
\label{LaguerrePolar}
p(z^{[k]},0^{[m-k]}) = (-1)^{k} \binom{m}{k}^{-1} \L^{(-m-1)}_{k}(z),
\end{equation}
where $p$ is the polar form of the polynomial $P$.
\end{proposition}
\begin{proof}
Let $P$ be an element of $\Pi_{m,n}$ and $k$ an integer such that $k \leq n$. Using the explicit 
expression of the polar form (\ref{ExplicitBlossom}), we obtain 
\begin{equation*}
p(z^{[k]},0^{[m-k]})  = \sum_{\ell=0}^{k}\frac{\binom{k}{\ell}}{\binom{m}{\ell} \ell!} z^{\ell}.
\end{equation*}
Thus, using the identity 
\begin{equation*}
\frac{\binom{k-m-1}{k-\ell}}{\binom{m}{k}} = (-1)^{k+\ell} \frac{\binom{k}{\ell}}{\binom{m}{\ell}}
\end{equation*} 
and comparing with the explicit expression of the generalized Laguerre polynomials 
(\ref{ExplicitLaguerre}) concludes the proof.\qed
\end{proof}
\section{Absolute stability radius}
\label{Sec3}
In this section, we use the theory of polar forms, Walsh's coincidence theorem and the results of 
the previous section, to give sharp upper bound on the absolute stability radius of explicit Runge-Kutta methods.
\subsection{Polar forms and Walsh's coincidence theorem}

A {\it circular region} of the complex plane is defined as one of the following: an open disc, 
a closed disc, an open half plane, a closed half plane, the open exterior of a circle 
or a closed exterior of a circle.  

\begin{theorem} ({\bf {Walsh coincidence Theorem}})
\label{walshM}
Let $P$ be a polynomial of exact degree $n \geq 1$. Let $u_1, u_2, \ldots, u_n$ 
be $n$ complex numbers which lie in a circular region $\mathcal{C}$. 
Then there exists a complex number $\zeta$ in $\mathcal{C}$ such that
\begin{equation*}
p(u_1,u_2,\ldots,u_n )=P(\zeta),
\end{equation*}
where $p$ is the polar form of the polynomial $P$.
\end{theorem} 

When the circular region $\mathcal{C}$ in Theorem \ref{walshM} is unbounded,
the condition that the polynomial $P$ is of exact degree $n$ can be relaxed to include
polynomials of degree at most $n$ \cite{aithaddou6}. We shall need the following straightforward consequence 
of Walsh's coincidence theorem.
\begin{corollary}
\label{WalshCorollary}
Let $P$ be a polynomial of degree at most $n \geq 1$ such that $|P(z)| \leq 1$ for any $z$ 
in a circular region $\mathcal{C}$.
Then for any complex numbers $u_1, u_2, \ldots, u_n$ in $\mathcal{C}$, we have 
\begin{equation*}
|p( u_1, u_2, \ldots, u_n )| \leq 1,
\end{equation*} 
where $p$ is the polar from of the polynomial $P$.
\end{corollary}
\begin{proof}
Let us assume that $P$ is of exact degree $k \leq n$. 
Suppose that there exist complex numbers $u_1, u_2, \ldots, u_k$ in
$\mathcal{C}$ such that $|p(u_1, u_2,\ldots, u_k)| > 1$. Then, by Walsh's 
coincidence theorem, there exists a $\zeta$ in $\mathcal{C}$ such that $|P(\zeta)|>1$. 
This leads to a contradiction. Therefore, for any $u_1, u_2, \ldots, u_k$ in $\mathcal{C}$, 
we have $|p(u_1, u_2,\ldots, u_k)|\leq 1$. When viewing $P$ as a polynomial of degree $n$, 
for any $u_1, u_2, \ldots, u_n$ in $\mathcal{C}$, we have 
\begin{equation*}
p(u_1, u_2, \ldots, u_n) = \binom{n}{k}^{-1}\sum_{\{i_1,\ldots,i_k\} \subset \{1,\ldots,n\}}
p(u_{i_1},u_{i_1},\ldots,u_{i_k}).
\end{equation*}	
where the sum is over all $k$-tuples $\{i_1,\ldots,i_k\} \subset \{1,\ldots,n\}$ of pairwise 
distinct integers. Thus, we clearly have $|p(u_1, u_2, \ldots, u_n)|\leq 1$.\qed
\end{proof}

\subsection{An upper Bound on the absolute stability radius}
Denote by $D_{r}$ the closed disc $D_{r} = \{z\in \mathbb{C} \; | \; |z + r| \leq r \}$. 
Given a polynomial $P$ in $\Pi_{m,n}$ we denote by
\begin{equation*}
r(P) = \sup \{r \; | \; D_{r} \subset S_{P}\},
\end{equation*}
where $S_{P}$ is the stability region of $P$. The absolute stability radius $r_{m,n}$ 
is defined by  
\begin{equation*}
r_{m,n} = \sup \{r(P) \; | \; {P\in \Pi_{m,n}} \}.
\end{equation*}
It is shown in \cite{owrenexistence} that there exists a unique polynomial $\Phi_{m,n}$ 
in $\Pi_{m,n}$ such that $r\left(\Phi_{m,n}\right) = r_{m,n}$ and that the stability 
region of $\Phi_{m,n}$ touches the circle $C_{m,n} = \{z\in \mathbb{C}
\; | \; |z + r_{m,n}| = r_{m,n}\}$ at least at $m-p+2$ distinct points. It is
conjectured in \cite{vander} that the optimal stability region touches 
the circle $C_{m,n}$ at exactly $m-p+2$ distinct points (see Figure \ref{fig:ComplexFig}).       

To exhibit the usefulness of the polar form, let us give a simple upper bound on $r_{m,n}$.
Denote by $\phi_{m,n}$ the polar form of the polynomial $\Phi_{m,n}$. The disc $D_{r_{m,n}}$ 
is a circular region and $|\Phi_{m,n}(z)| \leq 1$ for any $z \in D_{r_{m,n}}$. Thus by 
Corollary \ref{WalshCorollary}, we have 
\begin{equation*}
\left|\phi_{m,n}\left(-2r_{m,n},0^{[m-1]}\right)\right| \leq 1.
\end{equation*}
This is equivalent to stating that $|1-2r_{m,n}/m|\leq 1$. Thus, we arrive to the same result of  
Jeltsch and Nevanlinna \cite{nevanlinna}, i.e.; $r_{m,n} \leq m$. The previous arguments can be 
generalized to provide for a refined upper bound on $r_{m,n}$. More precisely, we have
\begin{theorem}
\label{ComplexBoundTheorem}
For any integers $1\leq n \leq m$, we have $r_{m,n} \leq -\xi/2$, where $\xi$ is the unique 
negative solution to the polynomial equation 
\begin{equation*}
\L^{(-m-1)}_{n}(x)-\binom{m}{n}  =0.
\end{equation*}
Moreover, $-\xi/2\leq  m-\frac{1}{2}\left(1 +(-1)^{n}\right)$.
\end{theorem}                                        
\begin{proof}
Denote by $\phi_{m,n}$ the polar form of $\Phi_{m,n}$. By Corollary \ref{WalshCorollary}, 
we have 
\begin{equation*}
\left|\phi_{m,n}\left(-2r_{m,n}^{[n]},0^{[m-n]}\right)\right| \leq 1.
\end{equation*}
Thus by Proposition \ref{PolarLaguerre}, we have
\begin{equation*}
\left| \L^{(-m-1)}_{n}(-2r_{m,n})\right| \leq \binom{m}{n}.
\end{equation*}
The claim of the theorem then follows easily form Corollary \ref{MainCorollary}.\qed
\end{proof}
Theorem \ref{ComplexBoundTheorem}, expressed in a different form, was also proved 
in \cite{owrenBIT} using the theory of positive functions.  
Table \ref{tab:tablecComplexBound} shows some values of the optimal stability radius $r_{m,n}$ and the
associated upper bounds derived from Theorem \ref{ComplexBoundTheorem}.
Figure \ref{fig:ComplexFig} shows the control polygon of the parametric curve $(t,\Phi_{5,3}(t))$ (left)
(resp. $(t,\Phi_{6,3}(t))$ (right) over the interval $[-2r_{5,3},0]$ (resp. $[-2r_{6,3},0]$). The stability region of these 
optimal stability polynomials is also shown. According to Corollary \ref{WalshCorollary}, the absolute value of the ordinate of each control point is less or equal to $1$.

\begin{table}[h!]
\centering
\normalsize{
\label{tab:tablecComplexBound}
\begin{tabular}{|c|c|} \hline
		$r_{m,n}$ &          upper bound    \\
		\hline
		$r_{4,3}$ =2.07       &  2.19         \\
		$r_{5,3}$ =2.94       &  3.15         \\
		$r_{6,4}$ =3.06       &  3.30         \\
		$r_{6,5}$ =2.37       &  2.53         \\
	    $r_{7,5}$ =3.17       &  3.47         \\
	   \hline
\end{tabular}}
\caption{Value of the stability radius $r_{m,n}$ and the upper bound given 
in Theorem \ref{ComplexBoundTheorem}, for various values of $m$ and $n$.}
\end{table} 
The weakness of our methodology in deriving upper bounds on the absolute stability radii is that the bounds does 
not take into account any extra information about the optimal stability polynomials beside the fact that they belong 
to $\Pi_{m,n}$. The polar form was instrumental in eliminating all the coefficients of order higher than $n$ without 
destroying the main characteristic of the initial polynomial. More refined bound can in principle 
be obtained by incorporating to our methodology extra-information on the optimal stability polynomials.

\begin{figure*}[h!]
\hskip 1 cm
\includegraphics[width=0.38\textwidth]{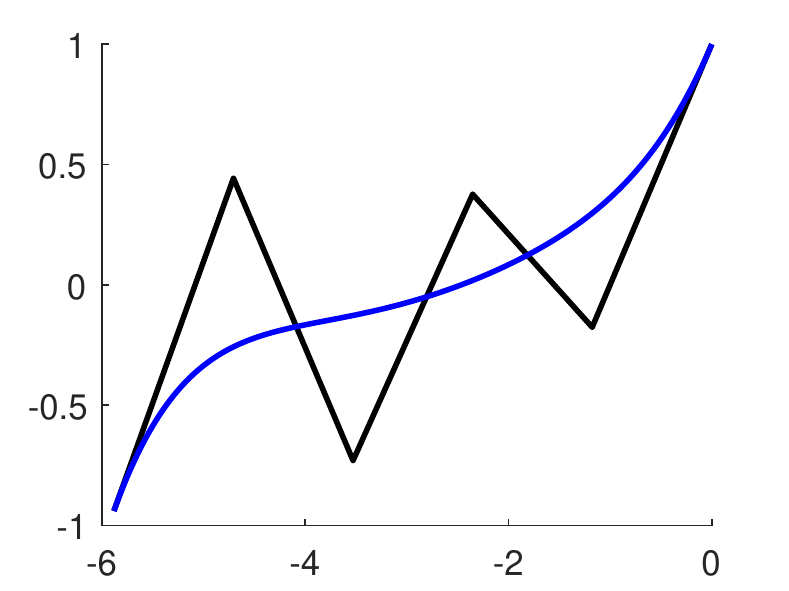}
\hskip 1 cm
\includegraphics[width=0.38\textwidth]{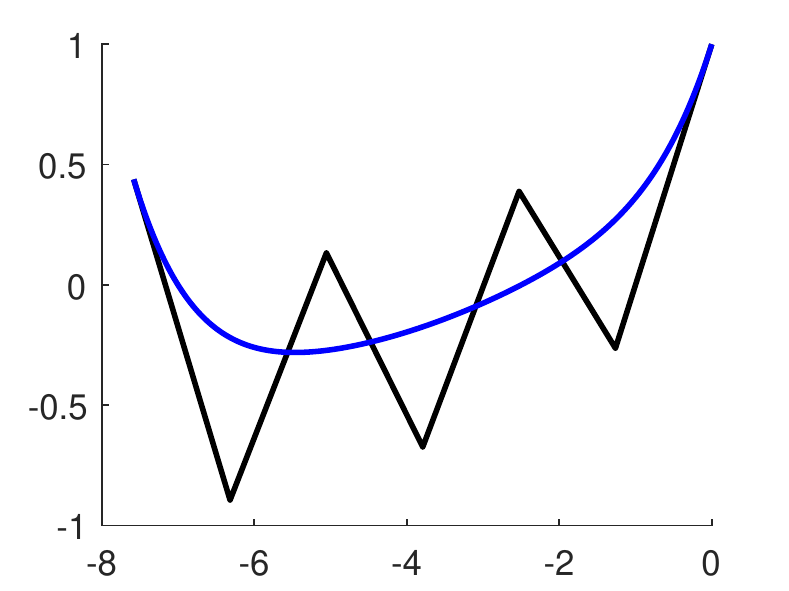}
\vskip -0.3 cm
\hskip 0.cm
\includegraphics[width=0.5\textwidth]{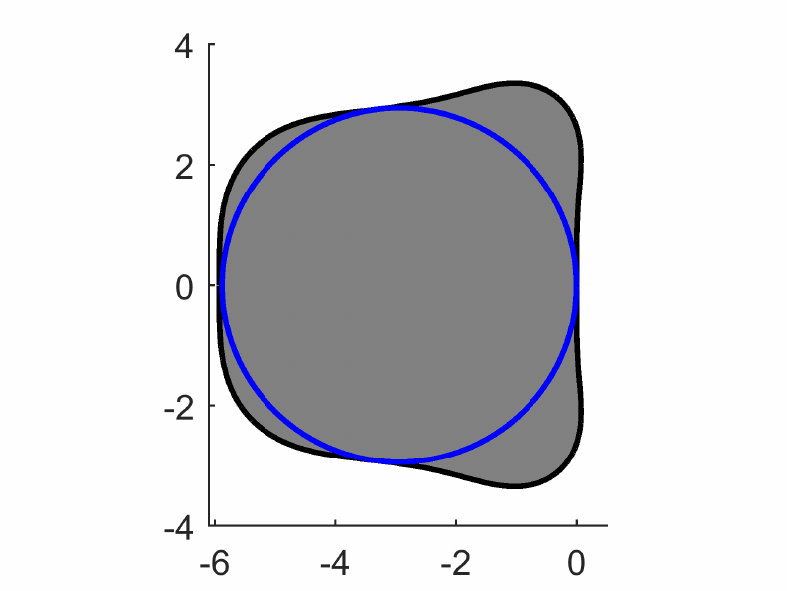}
\hskip -0.4cm
\includegraphics[width=0.5\textwidth]{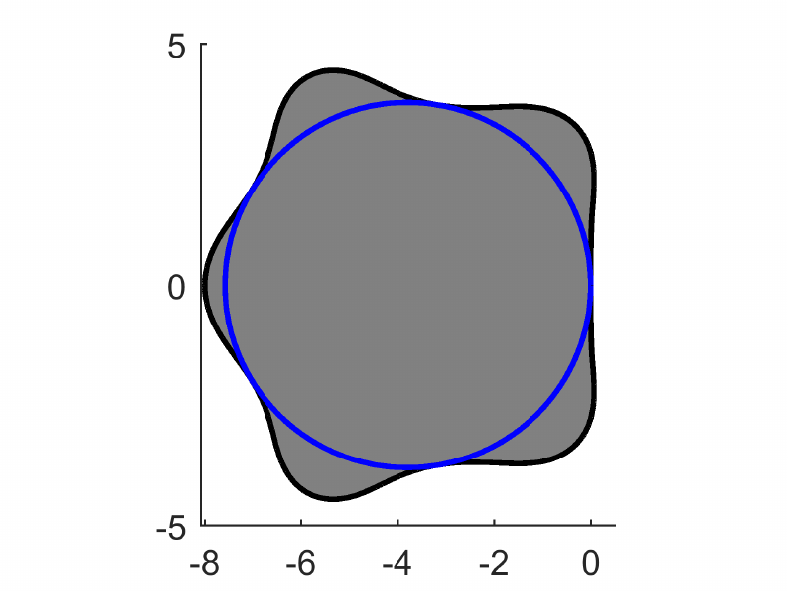}
\caption{The control polygon of the optimal stability polynomial $\Phi_{5,3}$ (left) (resp. $\Phi_{6,3}$ (right) )   
over the interval $[-2r_{5,3},0]$ (resp. $[-2r_{6,3},0]$) and their associated stability regions.}
\label{fig:ComplexFig}	
\end{figure*}

\subsection{Polar forms and stability results}
It is well-known that the optimal stability polynomial of first order and degree $m$ is
given by
\begin{equation}
\label{order1}
\Phi_{m,1}(z) = \left( 1 + \frac{z}{m}\right)^m.
\end{equation}
A first proof of this result was given by Jeltsch and Nevanlinna in \cite{nevanlinna}
using the theory of positive functions. A simpler proof, based on Bernstein's inequality 
and a comparison result between stability regions of two different methods with the 
same number of stages, was remarked in \cite{nevanlinna2}. Also, using the theory 
of positive functions, Owren and Seip showed that the second order optimal polynomials 
are given by \cite{owrenBIT}
\begin{equation}
\label{order2}
\Phi_{m,2}(z) = \frac{m-1}{m}\left( 1 + \frac{z}{m-1}\right)^m + \frac{1}{m}.
\end{equation}
They also pointed out in \cite{owrenBIT} that the existence and the uniqueness 
of the optimal polynomials, in conjunction with the sharpness of the bound in Theorem 
\ref{ComplexBoundTheorem} lead to simple proofs for the explicit expressions 
given in (\ref{order1}) and (\ref{order2}).    
In the following and based on the theory of polar forms, we give simple and self-contained 
proofs for the optimality of (\ref{order1}) and (\ref{order2}) without assuming the existence, 
the uniqueness or the bounds given in Theorem \ref{ComplexBoundTheorem}. 
There are two reasons for including these proofs in this paper. The first reason is that they are 
simple and can be understood with minimal knowledge. Second, 
these proofs are fundamental in understanding our strategy, in the next section, for deriving 
inequalities between the absolute stability radii of different methods.

We shall need the following simple lemma.

\begin{lemma}
\label{SimpleLemma}
Let $P$ be a complex polynomial of degree at most $n \geq 1$ such that 
\begin{equation}
\label{MiniPolar}
p(\lambda z^{[n-1]},0) = \alpha (1 + b z)^{n-1},
\end{equation}
where $\lambda, \alpha$ and $b$ are non-zero complex numbers and $p$ 
is the polar form of $P$. Then there exists a complex number $\beta$ such that
\begin{equation*}
P(z) = \alpha \left(1 + \frac{b}{\lambda}z\right)^{n} + \beta z^n. 
\end{equation*}
\end{lemma}
\begin{proof}
Let $b_1 =b, b_2,b_3,\ldots,b_{n}$ be complex numbers such that the functions
\begin{equation*} 
z^n, \left( 1+\frac{b_1}{\lambda}z\right)^{n}, \ldots, \left( 1+\frac{b_n}{\lambda}z\right)^{n}
\end{equation*}
form a basis of the space of polynomials of degree $n$. Writing $P$ in this basis
\begin{equation}
\label{PExpression}
P(z) = \beta z^n  + \sum_{i=1}^{n} \beta_i \left( 1+\frac{b_i}{\lambda}z\right)^{n},
\end{equation}
and evaluating (\ref{MiniPolar}), we obtain
\begin{equation*}
\sum_{i=1}^{n} \beta_i \left( 1+b_i z\right)^{n-1} = \alpha (1 + b_1 z)^{n-1}. 
\end{equation*}
Thus $\beta_1 = \alpha$ and $\beta_k =0$ for $k =2,3,\ldots,n$. Inserting these coefficients 
into (\ref{PExpression}) proves the lemma.\qed  
\end{proof}

\begin{proposition}
Let $P$ be a polynomial in $\Pi_{m,1}$. Then $D_{m} \subset S_{P}$ 
if and only if
\begin{equation}
\label{Euler1}
P(z) = \left( 1 + \frac{z}{m}\right)^m.
\end{equation}
\end{proposition}
\begin{proof}
Clearly if the polynomial $P$ is of the form (\ref{Euler1}) then $D_{m} \subset S_{P}$. 
Now, we proceed by proving the converse by induction on the integer $m$. 
The case $m=1$ being trivial, we assume the property to hold, up to degree $m-1$. 
Consider a polynomial $P$ in $\Pi_{m,1}$ such that $D_{m} \subset S_{P}$. Define 
the polynomial $Q$ by 
\begin{equation*}
Q(z) = p\left(\frac{m}{m-1}z^{[m-1]},0\right),
\end{equation*}
where $p$ is the polar form of $P$. The polynomial $Q$ belongs to $\Pi_{m-1,1}$ and 
by Corollary \ref{WalshCorollary}, we have $D_{m-1} \subset S_{Q}$. 
Thus, according to the induction hypothesis, we have 
\begin{equation*}
Q(z) = p\left(\frac{m}{m-1}z^{[m-1]},0\right) =  \left(1+\frac{z}{m-1}\right)^{m-1}.  
\end{equation*}
Therefore, according to Lemma \ref{SimpleLemma}, the polynomial $P$ is of the form 
\begin{equation*}
P(z) = \left( 1 + \frac{z}{m}\right)^m + \beta z^m.
\end{equation*}
It remains to show that $\beta = 0$ using the property that $D_{m} \subset S_{P}$.
By Corollary \ref{WalshCorollary}, we have 
\begin{equation*}
\left|p \left(-2m^{[m-2]}, m(i-1)^{[2]}\right)\right|^2 = 
1+\beta^2 2^{2m-2} m ^{2m} \leq 1. 
\end{equation*}
Thus $\beta =0$. This completes the proof.\qed
\end{proof} 

\begin{proposition}
Let $P$ be a polynomial in $\Pi_{m,2}$. Then $D_{m-1} \subset S_{P}$ 
if and only if
\begin{equation}
\label{Euler2}
P(z) = \frac{m-1}{m}\left( 1 + \frac{z}{m-1}\right)^m + \frac{1}{m}.
\end{equation}
\end{proposition}
\begin{proof}
Clearly if the polynomial $P$ is of the form (\ref{Euler2}) then $D_{m-1} \subset S_{P}$. 
Now, we proceed by proving the converse by induction on the integer $m$. 
The case $m=2$ being trivial, we assume the property to hold up to degree $m-1$. 
Given a polynomial $P$ in $\Pi_{m,2}$ such that $D_{m-1} \subset S_{P}$, 
we construct the polynomial
\begin{equation*}
Q(z) =\frac{m(m-2)}{(m-1)^2}p\left(\frac{m-1}{m-2}z^{[m-1]},0\right) + 
\left(1-\frac{m(m-2)}{(m-1)^2}\right),
\end{equation*}
where $p$ is the polar form of $P$. The polynomial $Q \in \Pi_{m-1,2}$ and 
according Corollary \ref{WalshCorollary} we have $D_{m-2} \subset S_{Q}$. 
Thus, by the induction hypothesis, the polynomial $Q$ is of the form 
\begin{equation*}
Q(z) =  \frac{m-2}{m-1}\left(1+\frac{z}{m-2}\right)^{m-1}+\frac{1}{m-1}.
\end{equation*}
This shows that 
\begin{equation*}
p\left(\frac{m-1}{m-2}z^{[m-2]},0\right)= 
\frac{m-1}{m}\left( 1 + \frac{z}{m-2}\right)^{m-1} + \frac{1}{m}.
\end{equation*}
Thus, according to Lemma \ref{SimpleLemma}, $P$ is of the form 
\begin{equation*}
P(z) = \frac{m-1}{m}\left(1+\frac{z}{m-1}\right)^{m} + \frac{1}{m} + \beta z^m.
\end{equation*}
To prove that $\beta=0$, we use the fact that $D_{m-1} \subset S_{P}$. By
Corollary \ref{WalshCorollary}, we have
\begin{equation*}
\left| p\left (-2(m-1)^{[m-2]}, (m-1)(i-1)^{[2]}\right)\right|^2 =
\left( \frac{1+ (-1)^m (m-1)}{m}\right)^2 + \beta^2(2m-2)^{2m} \leq 1.  
\end{equation*}
Thus $\beta = 0$.  This concludes the proof.\qed
\end{proof}

\subsection{Polar forms and inequalities between absolute stability radii}
In this section we generalize the idea behind the proof of the last two propositions
to provide for inequalities between the absolute stability radii of different methods.

Denote by $d\mu^{(m)}(x)$ the measure 
\begin{equation*}
d\mu^{(m)}(x) = \frac{m}{x^{m+1}} dx,
\end{equation*}
where $dx$ is the Lebesgue measure. The moments $\mathcal{M}_{k}$ of the measure $d\mu^{(m)}(x)$ 
over the interval $[1, \infty[$ are finite if and only if $k < m$. Moreover, we have 
\begin{equation}
\label{MomentFormula}
\mathcal{M}_k := \int_{1}^{\infty} x^{k} d\mu^{(m)}(x) = \frac{m}{m-k}, \quad k=0,1,\ldots,m-1.
\end{equation}
Thus, we can define orthogonal polynomials $\pi^{(m)}_{p}$ with respect to the measure $d\mu^{(m)}(x)$ 
of any degree $p$ such that $2p-1 < m$ by the formula    
\[
\pi^{(m)}_{p}(x) =  
\begin{vmatrix}
\mathcal{M}_{0} & \mathcal{M}_{1} & \ldots & \mathcal{M}_{p} \\
\mathcal{M}_{1} & \mathcal{M}_{2} & \ldots & \mathcal{M}_{p+1}   \\ 
\ldots & \ldots & \ldots &  \\
\mathcal{M}_{p-1} & \mathcal{M}_{p} & \ldots & \mathcal{M}_{2p-1} \\
1                & x & \ldots & x^{p}
\end{vmatrix}.
\]  
Denote by $\alpha_i, \lambda^{(m)}_{i,p}, i=1,2,\ldots,p,$ the weights and the nodes of 
the $p$-point Gaussian quadrature with respect to the measure $d\mu^{(m)}(x)$. 
Thus in particular, the real numbers $\lambda^{(m)}_{i,p}, i=1,2,\ldots,p,$ are the zeros 
of the orthogonal polynomial $\pi^{(m)}_{p}$ and for any polynomial $P$ of degree at most $2p-1$, 
we have 
\begin{equation}
\label{GaussFormula}
\int_{0}^{\infty} P(x) d\mu^{(m)}(x) = \sum_{i=1}^{p} \alpha_i P(\lambda^{(m)}_{i,p}). 
\end{equation}          

\begin{lemma}
\label{MomentProposition}
Let $F$ be an element of $\Pi_{m,n}$ and denote by $f$ its polar form. 
Let $p$ be an integer such that $2p-1 \leq n$. The polynomial $Q$ defined by 
\begin{equation*}
Q(z) = \sum_{i=1}^{p} \alpha_i f \left(\lambda^{(m)}_{i,p} z,
\lambda^{(m)}_{i,p} z,\ldots,\lambda^{(m)}_{i,p} z,0\right) 
\end{equation*}
is an element of $\Pi_{m-1,2p-1}$.
\end{lemma}
\begin{proof}
Invoking (\ref{MomentFormula}) and (\ref{GaussFormula}), for $k=0,1,\ldots,2p-1$, 
we have 
\begin{equation*}
Q^{(k)}(0) = \frac{m-k}{m} \sum_{i=1}^{p} \alpha_i \left[\lambda^{(m)}_{i,p}\right]^{k} 
= \frac{m-k}{m} \int_{1}^{\infty} x^k d\mu^{(m)}(x) = 1.
\end{equation*}
Therefore, the polynomial $Q$ is an element of $\Pi_{m-1,2p-1}$.\qed
\end{proof} 
As a consequence of Lemma \ref{MomentProposition}, we prove the following.
\begin{theorem}
For any integers $m, n$ and $p$ such that $m \geq n \geq 2p-1 \geq 1$, we have 
\begin{equation}
\label{SameInequality}
r_{m,n} \leq \lambda^{(m)}_p r_{m-1,2p-1}
\end{equation} 
where $\lambda^{(m)}_p$ is the largest zero of the orthogonal polynomial $\pi^{(m)}_{p}$. 
\end{theorem}
\begin{proof}
Let $\Phi_{m,n}$ the optimal stability polynomial with stability radius $r_{m,n}$. 
According to Lemma \ref{MomentProposition}, the polynomial $Q$ defined by  
\begin{equation*}
Q(z) = \sum_{i=1}^{p} \alpha_i \varphi_{m,n}
\left(\lambda^{(m)}_{i,p} z, \lambda^{(m)}_{i,p} z,\ldots,\lambda^{(m)}_{i,p} z,0\right), 
\end{equation*} 
where $\varphi_{m,n}$ the polar form of $\Phi_{m,n}$ is an element of $\Pi_{m,2p-1}$. 
Moreover, any complex number $z \in D_{r_{m,n}/ \lambda^{(m)}_p}$ satisfies 
\begin{equation*}
z \in D_{r_{m,n}/\lambda^{(m)}_{i,p}},
\quad \textnormal{for} \quad 
i=1,2,\ldots,p. 
\end{equation*} 
Thus according to Corollary \ref{WalshCorollary}, 
\begin{equation*}
|\varphi_{m,n} \left(\lambda^{(m)}_{i,p} z, \lambda^{(m)}_{i,p} z,\ldots,
\lambda^{(m)}_{i,p} z,0\right)| \leq 1
\quad \textnormal{for any} \quad 
z \in  D_{r_{m,n}/ \lambda^{(m)}_p}.  
\end{equation*}
Thus $|Q(z)| \leq \sum_{i=1}^{p} |\alpha_i| = 1$ for any 
$z \in  D_{r_{m,n}/ \lambda^{(m)}_p}$. This concludes the proof.\qed 
\end{proof}
An interesting special case of the previous theorem is the following inequality  
\begin{equation}
\label{NiceInequality}
r_{m-1,2p-1} \leq r_{m,2p-1} \leq \lambda^{(m)}_{p} r_{m-1,2p-1},
\quad  m > 2p-1.
\end{equation}
Thus for values of $\lambda^{(m)}_{p}$ that are close to $1$, the inequality (\ref{NiceInequality})
gives good bounds on the value of $r_{m,2p-1}$ as function of $r_{m-1,2p-1}$. 
Table \ref{tab:tableLambda} shows some of the value of $\lambda^{(m)}_{p}$ for various values of the 
parameters $m$ and $p$.  
\begin{table}
\centering
\normalsize{
	\label{tab:tableLambda}
	\begin{tabular}{|c|ccc|} \hline
	         &         & $p$    &           \\   
		$m$  & 2       & 3      &  4         \\
		\hline
		10   & 1.5000  & 2.3803 & 4.8984     \\
		30   & 1.1274  & 1.2577 & 1.4280     \\
		50   & 1.0730  & 1.1417 & 1.2243     \\
		70   & 1.0511  & 1.0977 & 1.1519     \\
		90   & 1.0393  & 1.0745 & 1.1148     \\
		\hline
	\end{tabular}}
	\caption{Values of $\lambda^{(m)}_{p}$ for various values of $m$ and $p$.}
\end{table}
\subsection{Inequalities between optimal threshold factors}
Threshold factors govern the maximally allowable step-size at which positivity 
or contractivity preservation of explicit one-step methods for initial value problems
is guaranteed \cite{kraaPoly,spijker}. The aim of this section is to show that the optimal 
threshold factors satisfy an inequality similar to (\ref{SameInequality}). Recall that a $C^{\infty}$ function 
$f$ is said to be {\it{absolutely monotonic}} over an interval $[a,b]$ if and only if, 
for any $x \in [a,b]$ and for any non-negative integer $k$, $f^{(k)}(x) \geq 0$. 
The {\it{threshold factor}}, $R(\phi)$, of a polynomial $\phi$ in $\Pi_{m,n}$ 
is defined by
\begin{equation*}
R(\phi) = \sup \{ r \; | \; r=0  \; \textnormal{or} \; ( r >0 \;  \textnormal{and} \; \phi \; \textnormal{is absolutely monotonic over }  [-r,0]) \}.
\end{equation*}
The {\it{optimal threshold factor}} $R_{m,n}$ is defined by 
\begin{equation*}
R_{m,n} =  \sup \{ R(\phi) / \phi \in \Pi_{m,n} \}.
\end{equation*}     
Kraaijevanger showed in \cite{kraaPoly} that $0 < R_{m,n} \leq m-n-1$ and that there exists 
a unique polynomial $\Psi_{m,n}$ in $\Pi_{m,n}$, called the {\it{optimal threshold polynomial}}, 
such that 
\begin{equation*}
R(\Psi_{m,n}) = R_{m,n}.
\end{equation*}
An improved inequality for $R_{m,n}$ is obtained in \cite{aithaddouThreshold} as 
\begin{equation*}
\ell_{p}^{(m-2p+1)} \leq R_{m,2p} = R_{m,2p-1} \leq \ell_{p}^{(m-p)},
\end{equation*}
where $\ell_{k}^{(n)}$ refers to the smallest zero of the generalized Laguerre 
polynomial $\L_{k}^{(n)}$. It is straightforward to show that for any positive integers 
$m,n$ such that $m \geq n$, we have 
\begin{equation}
\label{InterBound} 
R_{m,n} \leq r_{m,n}. 
\end{equation}
Indeed, to show (\ref{InterBound}), we write the optimal threshold polynomial $\Psi_{m,n}$ as 
\begin{equation*}
\Psi_{m,n}(z) = \sum_{k=0}^{m} \alpha_{k} \left( 1 + \frac{z}{R_{m,n}}\right)^k. 
\end{equation*}
The absolute monotonicity of $\Psi_{m,n}$ over $[-R_{m,n},0]$ implies the non-negativity 
of the coefficients $\alpha_k, k=0,1,\ldots,m$. Thus for any $z \in D_{R_{m,n}}$, we have 
\begin{equation*}
\left| \Psi_{m,n}(z) \right| \leq \sum_{k=0}^{m} \alpha_{k} = 1. 
\end{equation*}
One can view $R_{m,n}$ as a lower bound for $r_{m,n}$ and for which $R_{m,n}$  can be 
computed using the highly efficient algorithm described in \cite{aithaddouThreshold}.  
 
We shall need the following lemma.  
\begin{lemma}
\label{Prop:Monotone}
Let $P$ be a polynomial of degree at most $n$ and 
absolutely monotonic polynomial over a non-empty interval $[a,b]$. 
Then for any real numbers $(u_1,u_2,\ldots,u_n) \in [a,b]^{n}$ and for 
any $0 \leq k \leq n$ 
\begin{equation*}
p^{(k)}(u_1,u_2,\ldots,u_{n-k}) \geq 0
\end{equation*} 
where $p^{(k)}$ is the polar form of the $k^{em}$ derivative $P^{(k)}$ of the polynomial $P$.
\end{lemma}
\begin{proof}
We proceed by induction on $n$. The claim of the lemma being trivial for $n=1$, 
let us assume it holds up to degree $n$ polynomials. Let $P$ be a polynomial of degree 
at most $n+1$, absolutely monotonic over the interval $[a,b]$. 
Since for any integer $k$, $P^{(k)}$ is absolutely monotonic and due to the induction hypothesis, 
we only need to show that $p(u_1,u_2,\ldots,u_{n+1}) \geq 0$ for any $u_1,\ldots,u_n$ in $[a,b]$. 
Let us first assume that $u_1 < u_2 <\ldots<u_{n+1}$. We have
\begin{equation*}
(n+1) p'(u_1,u_2,\ldots,u_{n}) = \frac{p(u_1,u_2,\ldots,u_{n},u_{n+1}) 
- p(u_1,u_2,\ldots,u_{n},u_1) }{u_{n+1} - u_{1}} \geq 0. 
\end{equation*}
where $p'$ is the polar form of $P'$. Therefore, we have 
\begin{equation*}
p(u_1,u_2,\ldots,u_{n},u_{n+1}) \geq   p(u_1,u_2,\ldots,u_{n},u_1)
\end{equation*} 
Iterating the same argument, one can show that 
\begin{equation*}
\begin{split}
p(u_1,u_2,\ldots,u_{n},u_{n+1}) & \geq p(u_1,u_2,\ldots,u_{n},u_{1}) 
\geq p(u_1,u_2,\ldots,u_{n-1},u_{1},u_{1}) \\
&\geq p(u_1,u_2,\ldots,u_{n-2},u_{1},u_{1}) \geq \ldots \geq p(u_1,u_1,\ldots,u_{1},u_{1}) \geq 0.  
\end{split}
\end{equation*}
We conclude the proof using the symmetry of the polar form and a simple continuity argument.\qed
\end{proof}

\begin{theorem}
For any integers $m, n$ and $p$ such that $m \geq n \geq 2p-1 \geq 1$, we have 
\begin{equation*}
R_{m,n} \leq \lambda^{(m)}_p R_{m-1,2p-1}
\end{equation*} 
where $\lambda^{(m)}_p$ is the largest zero of the orthogonal polynomial $\pi^{(m)}_{p}$. 
\end{theorem}
\begin{proof}
Let $\Psi_{m,n}$ the optimal threshold polynomial with threshold factor $R_{m,n}$. 
According to Lemma \ref{MomentProposition}, the polynomial $Q$ defined by  
\begin{equation*}
Q(x) = \sum_{i=1}^{p} \alpha_i \psi_{m,n}
\left(\lambda^{(m)}_{i,p} x, \lambda^{(m)}_{i,p} x,\ldots,\lambda^{(m)}_{i,p} x,0\right), 
\end{equation*} 
where $\psi_{m,n}$ the polar form of $\Psi_{m,n}$ is an element of $\Pi_{m,2p-1}$. 
Moreover, according to Lemma\ref{Prop:Monotone}, each of the polynomials 
$\psi_{m,p} \left(\lambda^{(m)}_{i,p} x, \lambda^{(m)}_{i,p} x,\ldots,\lambda^{(m)}_{i,p} x,0\right)$
is absolutely monotonic over the interval $[-R_{m,n}/ \lambda^{(m)}_p,0]$. Thus, the polynomial $Q$ 
is absolutely monotonic over the interval $[-R_{m,n}/ \lambda^{(m)}_p,0]$. This concludes the proof. 
\qed
\end{proof}

\section{The parabolic stability radius}
\label{Sec4}
The semi-discretization of many parabolic partial differential equations leads to a system of ordinary differential 
equations for which the Jacobian matrix has negative eigenvalues. For instance, the semi-discretization in space 
of the diffusion equation  
\begin{equation*}
\partial_{t} u = D \partial_{xx} u, \quad x \in [a,b], \quad t > 0,
\end{equation*}
where $D$ is the diffusion coefficient, leads to the system of ordinary differential equations
\begin{equation}
\label{SemiDiscrete}
\partial_{t} u_i = D \frac{u_{i-1} - 2 u_{i} + u_{i+1}}{\Delta x^2}.
\end{equation} 
In the periodic case, the discretization of the Laplacian yields negative eigenvalues in the interval
$[-4D/\Delta x^2,0]$. Let $[-\beta,0]$ be the real segment which belongs to the stability region $S_{P}$ 
of a given Runge-Kutta method for solving (\ref{SemiDiscrete}). Thus the stability condition becomes 
\begin{equation}
\label{hSize}
h \leq \frac{\beta \Delta x^2 }{4 D}.
\end{equation} 
When the size $\Delta x$ of the grid is small, the interval containing all the eigenvalues is large and this in turn 
imposes the restrictive constraint (\ref{hSize}) on the step-size $h$ of the numerical scheme. 
To overcome this difficulty, one seek Runge-Kutta methods which have a stability boundary $\beta$ 
as large as possible. This motivates the following definitions.      
For a real polynomial $P$ in $\Pi_{m,n}$, denote by
\begin{equation*}
\theta(P) = \sup \{r >0 \; | \;  |P(x)| \leq 1 \quad \textnormal{for any}\quad  x\in [-r,0] \}.
\end{equation*}
{\it{The parabolic stability radius}} $\theta_{m,n} $ is defined by  
\begin{equation*}
\theta_{m,n} = \sup \{\theta(P) \; | \; P \in \Pi_{m,n} \}.
\end{equation*}
Riha showed in \cite{riha} that there exists a unique polynomial $\Theta_{m,n}  \in \Pi_{m,n}$
such that $\theta_{m,n} = \theta(\Theta_{m,n})$. We shall call the polynomial $\Theta_{m,n}$ the 
{\it{parabolic optimal polynomial}}. For linear first order methods, the parabolic stability radius
$\theta_{m,1}$ and its associated parabolic optimal polynomial $\Theta_{m,1}$ are well known 
\cite{franklin} and are given by 
\begin{equation}
\label{Tpoly}
\theta_{m,1} = 2 m^2  
\quad \textnormal{and} \quad 
\Theta_{m,1}(x) = T_{m} \left(1 + \frac{x}{m^2} \right),
\end{equation}
where $T_{m}$ are the classical Chebyshev polynomials. The parabolic optimal polynomials for linear 
second order methods can be expressed in terms of Zoltarev polynomials which themselves are explicitly 
expressed in terms of elliptic functions \cite{lebedev}. No close analytical expressions for the
parabolic optimal polynomials of methods of order greater than $2$ are known 
and numerical methods are usually required for their computations \cite{lebedev2} 
(Figure \ref{fig:RealFig} shows some high order parabolic optimal polynomials, 
their control polygon structures and the associated stability regions). 

In the following, we provide for upper and lower bounds on the parabolic stability 
radius $\theta_{m,p}$. 

Denote by $T_{m}(x,[a,b])$ the shifted Chebyshev polynomials over the interval $[a,b]$, i.e.; 
\begin{equation*}
T_{m}(x,[a,b]) = T_{m}\left(\frac{2x - (b+a)}{b-a}\right). 
\end{equation*}
The expression of the polynomials $T_{m}(.,[a,b])$ in terms of the Bernstein basis over $(a,b)$ 
is given by \cite{popa}  
\begin{equation}
\label{explicitT}
T_{m}(x,[a,b]) = \sum_{i=0}^{m} (-1)^{m-i} \frac{\binom{2m}{2i}}{\binom{m}{i}} B_{i}^{m}(x,[a,b]).
\end{equation}
Let us recall the following result of Lubinsky and Ziegler \cite{lubinsky}.

\begin{theorem}
Let $P$ be a real polynomial of degree at most $m \geq 1$ such that $|P(x)| \leq 1$ for any $x \in [a,b]$. Then 
\begin{equation}
\label{LubinskyInequality}
|p(a^{[k]}, b^{[m-k]}) |\leq |t_{m}(a^{[k]}, b^{[m-k]})|, \quad \textnormal{for} \quad k=0,1,\ldots,m,  
\end{equation}
where $p$ (resp. $t_{m}$) the polar form of the polynomial $P$ (resp. $T_{m}(.,[a,b])$).
\end{theorem}
To exhibit the usefulness of the previous theorem, let us apply it to the parabolic 
optimal polynomial $\Theta_{m,n}$. Using (\ref{LubinskyInequality}) with 
$[a,b] = [-\theta_{m,n},0]$ and $k=1$, we obtain 
\begin{equation*}
\left|\vartheta_{m,n}(-\theta_{m,n},0^{[m-1]})\right| = \left| 1 - 
\frac{\theta_{m,n}}{m}\right| \leq 2m-1,
\end{equation*}
where $\vartheta_{m,n}$ the polar form of $\Theta_{m,n}$. Therefore, we have 
$\theta_{m,n} \leq 2 m^2$. This inequality is sharp and is attained by 
the parabolic optimal polynomial $\Theta_{m,1}$ given in (\ref{Tpoly}).

By eliminating, via the polar form, all the terms of order larger than $n$ in the 
parabolic optimal polynomial  $\Theta_{m,n}$ and using the Lubinsky-Ziegler's inequality 
(\ref{LubinskyInequality}) we obtain the following.   
\begin{theorem}
\label{BoundTheoremreal}
For any integers $1\leq n \leq m$, we have $\theta_{m,n} \leq -\xi$, where $\xi$ is unique 
negative solution to the polynomial equation 
\begin{equation*}
\L^{(-m-1)}_{n}(x)-{\binom{2m}{2n}}=0.
\end{equation*}
Moreover, we have the following upper bound 
\begin{equation}
\label{BoundR1}
\theta_{m,n} \leq (n!)^{1/n} \left( \binom{2m}{2n} - 
\binom{m}{n}\right)^{1/n} + 2m - (1 + (-1)^n). 
\end{equation}
\end{theorem}                                        
\begin{proof}
Since the polynomial $\Theta_{m,n}$ satisfies $|\Theta_{m,n}(x)| \leq 1$ for any $x \in [-\theta_{m,n},0]$, 
by Lubinsky-Ziegler's inequality (\ref{LubinskyInequality}), we have 
\begin{equation*}
\left| \vartheta_{m,n} \left(-{\theta_{m,n}}^{[n]},0^{[m-n]} \right) \right| 
\leq  \frac{\binom{2m}{2n}}{\binom{m}{n}}.
\end{equation*}
Thus using (\ref{LaguerrePolar}), we obtain 
\begin{equation}
\label{secondTerm}
\left|\L_{n}^{(-m-1)}(-\theta_{m,n})\right| \leq \binom{2m}{2n}.
\end{equation}
Since the right hand side of (\ref{secondTerm}) is larger than $\binom{m}{n}$, 
we conclude the proof using Corollary \ref{MainCorollary}.\qed
\end{proof}
There are many numerical evidences suggesting that the quantity $\theta_{m,n}/m^2$ increases 
as a function of $m$, and therefore should converge to a finite limit. 
However, as far as we know, a proof of this fact is still missing. 
Using the previous theorem, we give an upper bound of this limit provided that it exists. 
\begin{corollary}
For any positive integer $n\geq 1$, we have
\begin{equation}
\label{LimitOptimal}
\lim_{m \to \infty} \frac{\theta_{m,n}}{m^2} \leq 
4 \left(\frac{n!}{(2n)!} \right)^{1/n}
\end{equation}
provided that the limit in the left-hand side of (\ref{LimitOptimal}) exists.
\end{corollary}
\begin{proof}
The inequality (\ref{LimitOptimal}) is a direct consequence of the inequality 
(\ref{BoundR1}) and the limits 
\begin{equation*}
\lim_{m \to \infty} \frac{\binom{2m}{2n}^{1/n}}{m^{2}} = \frac{4}{(2n)!} 
\quad \textnormal{and} \quad 
\lim_{m \to \infty} \frac{\binom{m}{n}}{m^{2n}} = 0.
\end{equation*} \qed
\end{proof}
Comparison between the bound given in the previous corollary and the existing 
numerical limits are given as follows:
\begin{equation*}
\lim_{m \to \infty} \frac{\theta_{m,2}}{m^2} \simeq 0.82 \leq 1.15, 
\quad 
\lim_{m \to \infty} \frac{\theta_{m,3}}{m^2} \simeq 0.50 \leq 0.81,
\quad 
\lim_{m \to \infty} \frac{\theta_{m,4}}{m^2} \simeq 0.35 \leq 0.62. 
\end{equation*}  

\begin{figure*}[h!]
\includegraphics[width=0.45\textwidth]{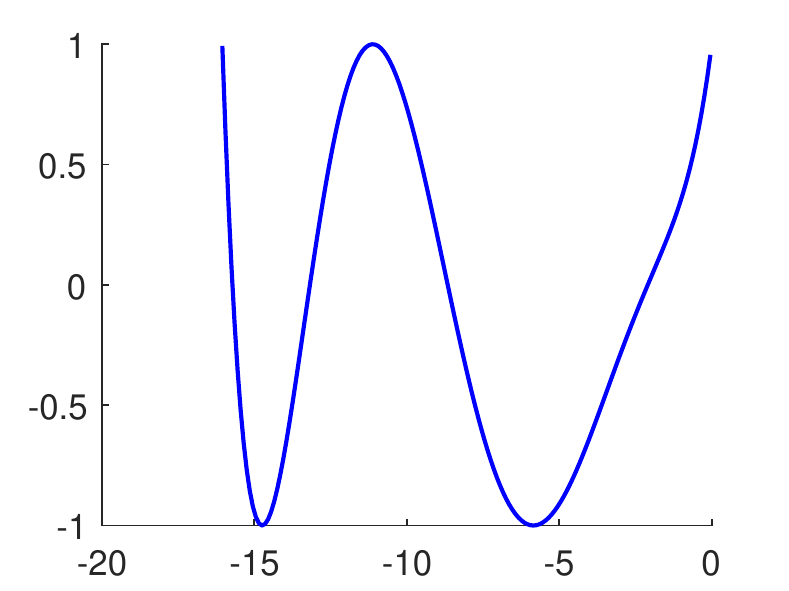}
\hskip 1 cm
\includegraphics[width=0.45\textwidth]{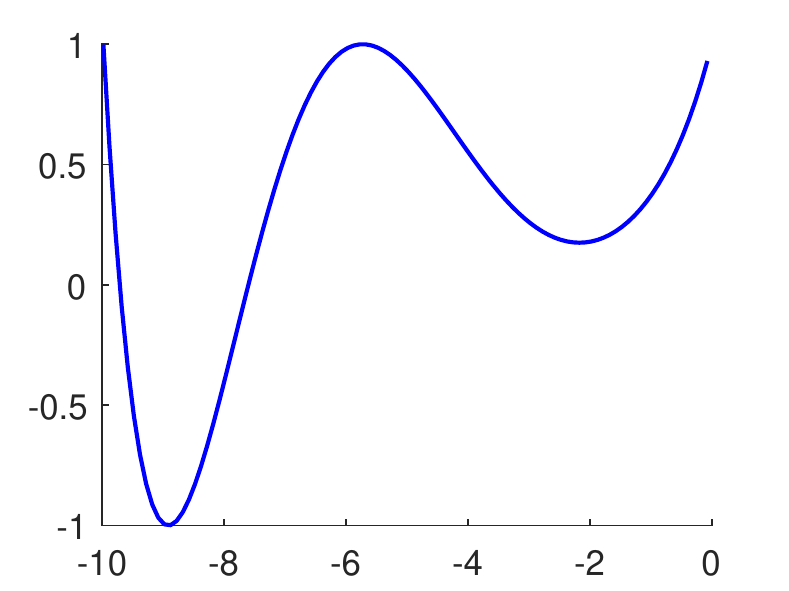}
\vskip 0cm
\includegraphics[width=0.45\textwidth]{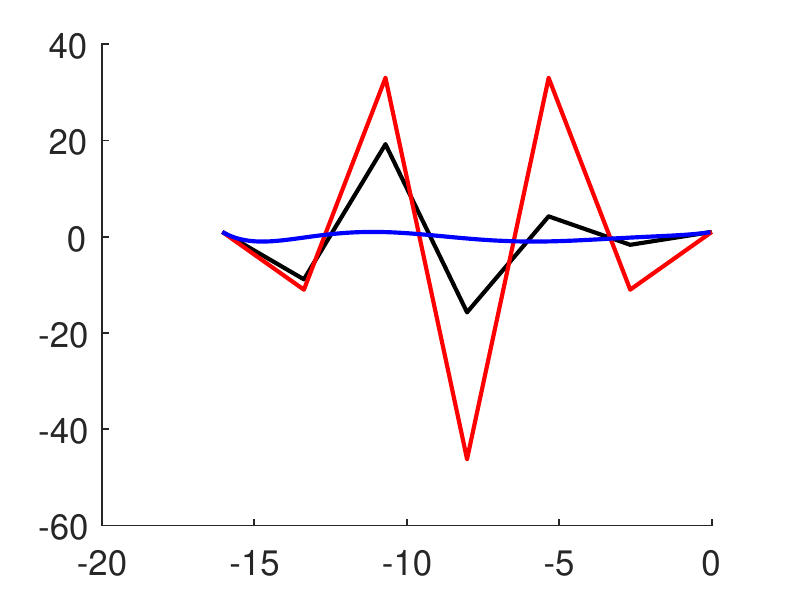}
\hskip 1 cm
\includegraphics[width=0.45\textwidth]{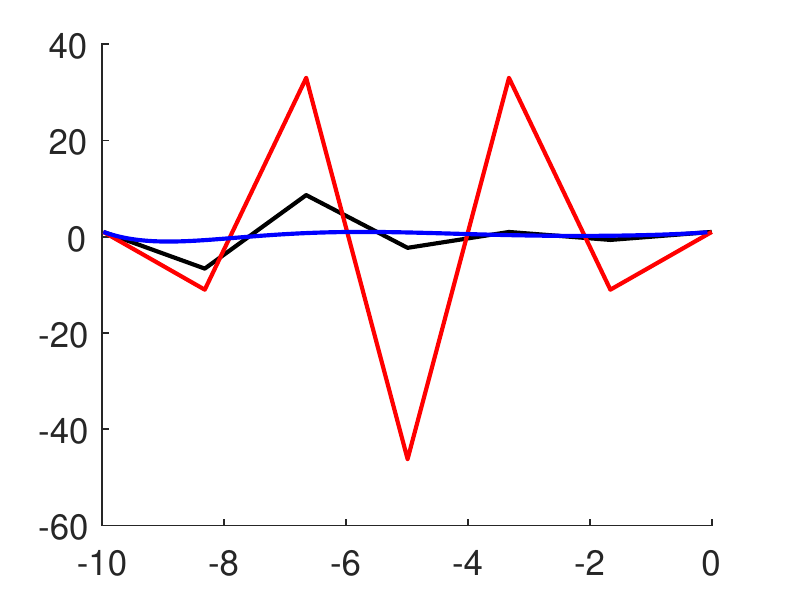}
\vskip 0. cm
\hskip 0 cm
\includegraphics[width=0.46\textwidth]{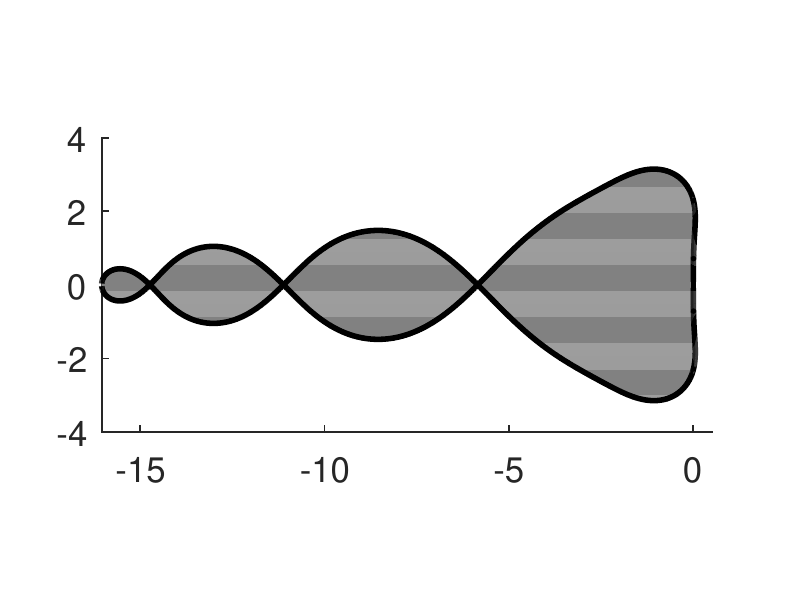}
\hskip 1.0cm
\includegraphics[width=0.46\textwidth]{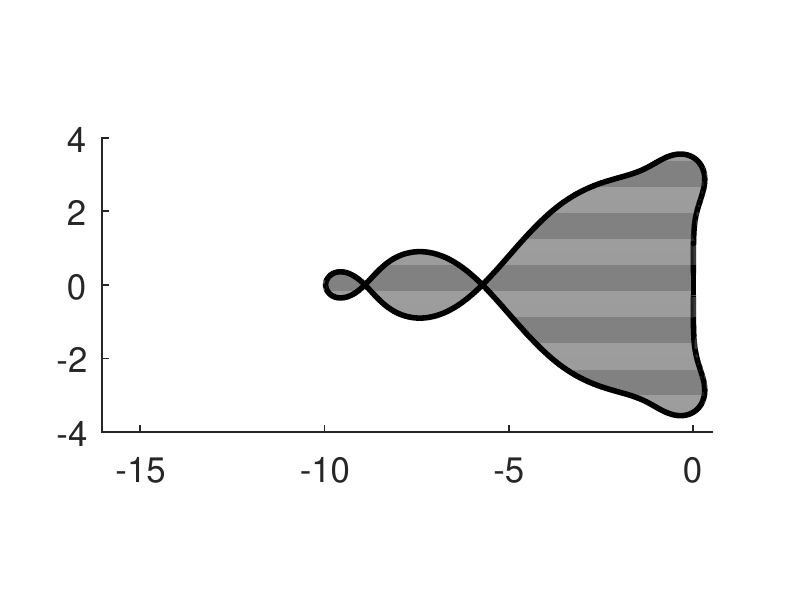}
\vskip -0.5 cm
\caption{The optimal parabolic polynomials $\Theta_{6,3}$ (left) and $\Theta_{6,4}$ (right), 
their control polygons and their associated stability regions. Red control polygons refer to the control 
polygons of Chebyshev polynomials over the associated intervals. Lubinsky and Ziegler's inequality  (\ref{LubinskyInequality}) states that the absolute value of the ordinate of each control point of the optimal stability polynomial (black color) is smaller than the ordinate of the associated control point of the Chebyshev polynomial (red color).}
\label{fig:RealFig}	
\end{figure*}

\subsection{Lower bound for parabolic stability radius}
In this section, we provide for a lower bound on the parabolic stability radius $\theta_{m,n}$.
An obvious upper bound is given by the inequality $\displaystyle{\theta_{m,p} \geq \theta_{m,m}}$, where 
the quantity $\theta_{m,m}$ is obviously given by the only negative zero of the polynomial equation 
$\L_{m}^{(-m-1)}(x) - 1 = 0$. Here, we show that we can improve on this lower bound by using a refined 
variation diminishing property of B\'ezier curves proved in \cite{aithaddou5} . 

For a non-zero degree $n$ polynomial $F$, we denote by $Z^{n}_{[a,b]}(F)$ the number of real zeros of $F$ (counting 
multiplicities). For a finite sequence of real numbers $r =(r_{1},r_{2},\ldots,r_{n})$ we denote 
by $S(r)$ the number of changes in sign in the sequence $r$ counting zero as a sign change. 
We denote by $S_{R}(r) = S(r_{1},r_{2},\ldots,r_{n}) - S(r_{1})$. We recall the following result proved in \cite{aithaddou5}.
\begin{theorem}
\label{VariationDiminishing}
For any $0 < \ell \leq n$, we have 
\begin{equation*}
Z^{n}_{[a,b]} \left(\sum_{k=0}^{n} q_{k} B_{k}^{n}\right) \leq Z^{n-\ell}_{[a,b]} 
\left(\sum_{k=0}^{n-\ell} q_{k} B_{k}^{n-\ell}\right) + 
S_{R}\left(q_{n-\ell+1},q_{n-\ell+2},\ldots,q_{n}\right),
\end{equation*} 
where $(B_{0}^{n},\ldots,B_{n}^{n})$ the Bernstein basis over $(a,b)$. 
\end{theorem}
Based on this theorem, we show the following.
\begin{proposition}
\label{Prop:Less1}
Let $Q(x) = \sum_{k=0}^{n} q_{k} B_{k}^{n}(x)$ be a polynomial of degree at most $n \geq 1$ such that
$|Q(x)| \leq 1$ for any $x \in [a,b]$.  Then, for any $m \geq n$, the polynomial 
$Q_{m}(x) =　\sum_{k=0}^{n} q_{k} B_{k}^{m}(x)$ satisfies $|Q_{m}(x)| \leq 1$ for any $x \in [a,b]$. 
\end{proposition}
\begin{proof}
 Let $\epsilon >0$ and denote by $F$ (resp. $F_{m}$) the polynomial $F(x)=Q(x)-1-\epsilon$ 
(resp. $F_{m}(x) = Q_{m}(x) - 1 - \epsilon$). Note that $F(x) < 0$ for any $x \in [a,b]$. We have 
\begin{equation*}
F(x) = \sum_{k=0}^{n} (q_{k} -1 -\epsilon)  B_{k}^{n}(x)
\end{equation*}
and
\begin{equation*}
F_{m}(x) = \sum_{k=0}^{n} (q_{k} -1 -\epsilon)  B_{k}^{m}(x) + \sum_{k=n+1}^{m} (-1 -\epsilon)  B_{k}^{m}(x).
\end{equation*}
From Theorem \ref{VariationDiminishing}, we have 
\begin{equation*}
Z_{[a,b]}^m(F_{m}) \leq Z_{[a,b]}^n(F) + S_{R}(-1 -\epsilon,-1 -\epsilon,\ldots,-1 -\epsilon) = 0.
\end{equation*}
Thus $F_{m}$ has no zero in the interval $[a,b]$. In other words, for any $\epsilon>0$ and any $x \in [a,b]$, 
we have $Q_{m}(x) < 1+ \epsilon$. Therefore, $Q_{m}(x) \leq 1$ for any $x \in [a,b]$. Similar arguments with 
the polynomial $F(x) = Q(x)+1+\epsilon$ concludes the proof. \qed   
\end{proof}

We are now in a position to give a lower bound on the parabolic stability radius $\theta_{m,n}$ . 
\begin{theorem}
\label{realLower}
For any integers $1 \leq n \leq m$, the parabolic stability radius $\theta_{m,n} \geq -\eta$, 
where $\eta$ is the unique negative zero of the polynomial equation 
\begin{equation}
\label{LowerEquation}
\L_{n}^{(-m-1)}(x) - \binom{m}{n} = 0.
\end{equation}
\end{theorem}
\begin{proof}
Let $\eta$ be the unique negative zero to the polynomial equation (\ref{LowerEquation}). 
Define the polynomial $Q$, in terms of the Bernstein basis over $(0,\eta)$, by
\begin{equation*}
Q(x) = \sum_{k=0}^n q_{k} B_{k}^{n}(x) 
\quad \textnormal{with} \quad 
q_{k} = (-1)^k \binom{m}{k}^{-1} \L_{k}^{(-m-1)}(\eta).
\end{equation*}
Denote by $Q_{m}$ the polynomial of degree at most $m$ defined by  
\begin{equation*}
Q_{m}(x) = \sum_{k=0}^p q_{k} B_{k}^{m}(x).
\end{equation*}   
Using Proposition \ref{PolarLaguerre}, it is clear that the polynomial $Q_{m}$ 
belongs to $\Pi_{m,n}$. Moreover, we have 
\begin{equation*}
\left| Q(x)\right| = \left| q_{m}(x^{[n]},0^{[m-n]}) \right| = 
\left| \binom{m}{n}^{-1} \L_{n}^{(-m-1)}(x) \right| \leq 1
\quad \textnormal{for any} \quad 
x\in [\eta,0]
\end{equation*}
Thus, by Proposition \ref{Prop:Less1}, $|Q_{m}(x)|\leq 1$ for any $x\in [\eta,0]$ and thus
$\theta_{m,n} \geq -\eta$. \qed
\end{proof}
Table \ref{tab:table5} shows some exact values of $\theta_{m,n}$ with
the upper and lower bounds derived from Theorem \ref{BoundTheoremreal} 
and Theorem \ref{realLower}.   
\begin{table}[h!]
	\label{tab:table5}
	\begin{tabular}{|c|ccc|} \hline
		\small{Stages}  & \hskip -0.2cm  \small{Lower bound} & \hskip -0.5cm  $\theta_{m,3}$  & \hskip -0.5cm \small{Upper bound}\\
		\hline
		4    &  \hskip -0.2cm  4.39    & \hskip -0.5cm  6.027   &  \hskip -0.5cm 7.202\\
		5    &  \hskip -0.2cm  6.308   & \hskip -0.5cm  10.535  &  \hskip -0.5cm  13.541\\
		6    &  \hskip -0.2cm  8.253   & \hskip -0.5cm  16.045  &  \hskip -0.5cm 21.481\\
		7    &  \hskip -0.2cm  10.214  & \hskip -0.5cm  22.56   &  \hskip -0.5cm  31.03\\
		8    &  \hskip -0.2cm  12.186  & \hskip -0.5cm  30.074  &  \hskip -0.5cm  42.193\\
		9    &  \hskip -0.2cm  14.163  & \hskip -0.5cm  38.596  &  \hskip -0.5cm  54.971\\
		10   &  \hskip -0.2cm  16.146  & \hskip -0.5cm  48.11   &  \hskip -0.5cm  69.367\\
		11   &  \hskip -0.2cm  18.132  & \hskip -0.5cm  58.637  &  \hskip -0.5cm  85.382\\
		12   &  \hskip -0.2cm  20.12   & \hskip -0.5cm  70.171  &  \hskip -0.5cm  103.02\\
		\hline
	\end{tabular}
    \hskip 0.3 cm
	\begin{tabular}{|c|ccc|} \hline 
		\small{Stages} & \hskip -0.2cm \small{Lower bound} & \hskip -0.5cm $\theta_{m,4}$ & \hskip -0.5cm \small{Upper bound}\\
		\hline
		5   & \hskip -0.2cm 4.688   & \hskip -0.5cm 6.06   & \hskip -0.5cm 7.32  \\
		6   & \hskip -0.2cm 6.600   & \hskip -0.5cm 9.972  & \hskip -0.5cm 13.063\\
		7   & \hskip -0.2cm 8.528   & \hskip -0.5cm 14.592 & \hskip -0.5cm 20.048\\
		8   & \hskip -0.2cm 10.471  & \hskip -0.5cm 19.929 & \hskip -0.5cm 28.275\\
		9   & \hskip -0.2cm 12.424  & \hskip -0.5cm 25.976 & \hskip -0.5cm 37.743\\
		10  & \hskip -0.2cm 14.385  & \hskip -0.5cm 32.74  & \hskip -0.5cm 48.455\\
		11  & \hskip -0.2cm 16.352  & \hskip -0.5cm 40.22  & \hskip -0.5cm 60.411\\
		12  & \hskip -0.2cm 18.324  & \hskip -0.5cm 48.412 & \hskip -0.5cm 73.614\\
		13  & \hskip -0.2cm 20.063  & \hskip -0.5cm 57.324 & \hskip -0.5cm 88.063\\
		\hline
	\end{tabular}
	\caption{Upper and lower bounds on $\theta_{m,n}$ for  $n=3$ and $n=4$ 
		     and different number of stages.}
\end{table}

\subsection{The parabolic stability radius with damping}
The parabolic optimal polynomial $\Theta_{m,1} = T_{m}(., [-2m^2,0])$ has the property that its 
stability region, at the values $x_i$ where $|\Theta_{m,1}(x_i)| = 1$, has zero width 
(see Figure \ref{fig:DampingFig}, left and Figure \ref{fig:RealFig} ). This represents an inconvenience when dealing
with the semi-discretization of many hyperbolic-parabolic partial differential equations that
exhibits a spectrum of the Jacobian that is contained on a strip around the negative axis. 
To overcome this difficulty, Guillou and Lago \cite{guillou} suggested replacing the stability 
requirement $|P_{m}(x)| \leq 1$ for $x \in [-\beta,0]$ by $|P_{m}(x)| \leq 1 - \eta <1$ for 
$x \in [-\beta_{\eta},-\delta_{\eta}]$ where $\delta_{\eta}$ is small parameter depending on $\eta$.
For $\delta = 0$, this amounts to replacing the polynomial $\Theta_{m,1}$ by the polynomial 
\begin{equation*}
\bar{\Theta}_{m,1}(x) = \frac{1}{T_{m}(\omega_0)} T_{m} \left(\omega_0 + \omega_1 x\right), 
\quad \omega_0 = 1 + \frac{\eta}{m^2}, \quad \omega_1 = \frac{T_m(\omega_0)}{T'_{m}(\omega_0)}. 
\end{equation*}
This way, even though the stability region along the negative axis becomes a little bit shorter,
a strip around the negative real axis is included in the stability region (see Figure \ref{fig:DampingFig}, right).
By increasing the value of $\eta$, larger strips around the negative real axis could be included in the 
stability region. This property, called {\it{damping}}, has been implemented for the construction
of explicit stabilized Runge-Kutta methods in \cite{torrilhon,verwer}. 
\begin{figure*}[h!]
\includegraphics[width=0.5\textwidth]{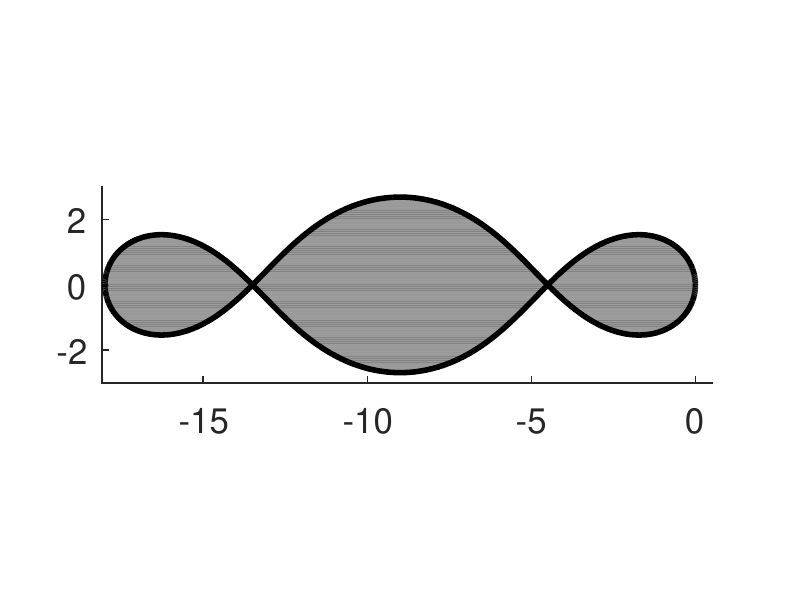}
\includegraphics[width=0.5\textwidth]{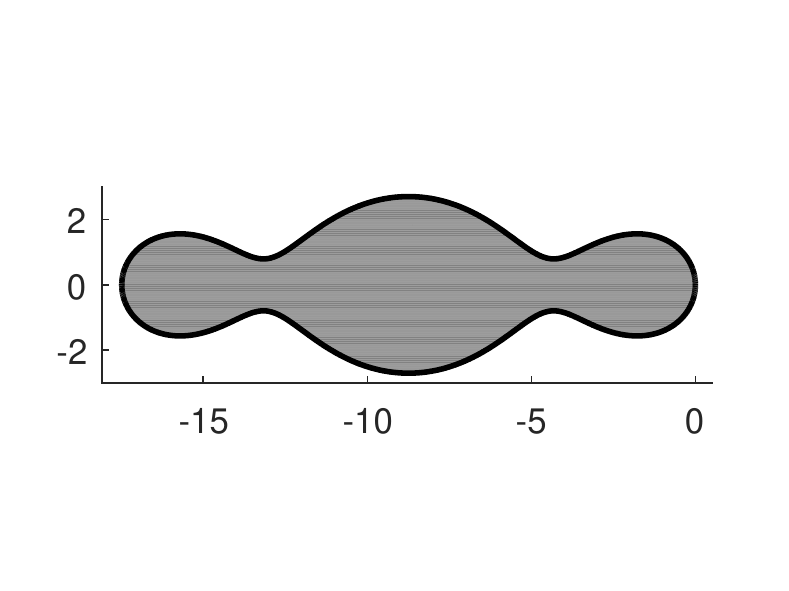}
\vskip -1.3 cm
\caption{Stability regions for shifted Chebyshev polynomials of degree 3. Undamped polynomial (left
Figure) and damped polynomial with $\eta =0.05$.}
\label{fig:DampingFig}	
\end{figure*}
To generalize damping to higher order methods, we introduce the following definitions: 
Given two real numbers $0 \leq \eta < 1$ and $\delta \geq 0$, and a polynomial $P$ in $\Pi_{m,n}$, we denote by  
\begin{equation*}
\ell(P, \eta,\delta) = \sup \{r>0 \; | \;  |P(x)| \leq 1-\eta 
\quad \textnormal{for any} \quad x \in [-r,-\delta] \}.
\end{equation*}   
We define the optimal damped radius $\ell_{m,n}(\eta,\delta)$ by
\begin{equation*}
\ell_{m,n}(\eta,\delta) = \sup \{\ell(P, \eta,\delta) | \;  P \in \Pi_{m,n} \}.
\end{equation*} 
Using the methodology introduced in \cite{riha}, it is not difficult to show that for 
any $0\leq \eta \leq 1$ and $\delta \geq 0$, there exists a unique polynomial $\bar{\Phi}_{m,n} \in \Pi_{m,n}$ 
such that $\ell(\bar{\Phi}_{m,n},\eta,\delta) = \ell_{m,n}(\eta,\delta)$.
  
From now on, given a real interval $[a,b]$, we denote by $\mathbb{D}_{[a,b]}$ the open disc 
with diameter $[a,b]$. To give an upper bound on the quantity $\ell_{m,n}(\eta,\delta)$, 
we recall the following result proved by Bernstein in \cite{bernstein} and rediscovered by 
Erd\"os in \cite{erdos}. 
\begin{theorem}
\label{Erdos}
Let $P$ be a real polynomial of degree at most $n \geq 1$ and let $[a,b]$ be a given interval. 
For any complex number $z \in \mathbb{C} \setminus \mathbb{D}_{[a,b]}$, we have
\begin{equation*}
|P(z)|  \leq |T_{n}(z; [a,b])| \; ||P||_{[a,b]},
\end{equation*} 
with equality if and only if the polynomial $P$ is a constant multiple of $T_{n}(.; [a,b])$.
\end{theorem} 
Combining Theorem \ref{Erdos} with Walsh's coincidence theorem, one can deduce 
the following. 

\begin{corollary}
\label{GeneralLubinsky}
Let $P$ be a polynomial of degree at most $n \geq 1$ and $[a,b]$ be a given real interval. For any 
$z_1,z_2,\ldots,z_n$ in $\mathbb{C} \setminus \mathbb{D}_{[a,b]}$, we have 
\begin{equation}
\label{WalshChebyshev}
|p(z_1,z_2,\ldots,z_n)| \leq |t_{n}(z_1,z_2,\ldots,z_n)| \; ||P||_{[a,b]} 
\end{equation}  
where $p$ (resp. $t_n$) the polar form of the polynomial $P$ (resp. $T_n(.,[a,b]$).
\end{corollary} 
\begin{proof}
We first note that for any $z_1,z_2,\ldots,z_n$ in $\mathbb{C} \setminus \mathbb{D}_{[a,b]}$, we have 
$t_{n}(z_1,z_2,\ldots,z_n) \not= 0$. Otherwise, by Walsh's coincidence theorem, one can find a complex number   
$\xi \in \mathbb{C} \setminus \mathbb{D}_{[a,b]}$ such that $T_{n}(\xi; [a,b]) = 0$. 
This contradicts the fact that all the zeros of $T_{n}(., [a,b])$ are reals and lie in the open interval $]a,b[$.
Let us assume that there exist $z_1,z_2,\ldots,z_n$ in $\mathbb{C} \setminus \mathbb{D}_{[a,b]}$ for which 
(\ref{WalshChebyshev}) does not hold. Thus, there exists a complex number $\omega$ with $|\omega| >1$ and  
such that 
\begin{equation*}
p(z_1,z_2,\ldots,z_n) - \omega t_{n}(z_1,z_2,\ldots,z_n) ||P||_{[a,b]} = 0.
\end{equation*}
Therefore, by Walsh's coincidence theorem, there exists a $\xi \in \mathbb{C} \setminus \mathbb{D}_{[a,b]}$
such that 
\begin{equation*}
P(\xi) - \omega T_{n}(\xi,[a,b])  ||P||_{[a,b]} = 0,
\end{equation*} 
or equivalently, there exists a $\xi \in \mathbb{C} \setminus \mathbb{D}_{[a,b]}$ such that 
\begin{equation*}
|P(\xi)|  > |T_{n}(\xi,[a,b])| \; ||P||_{[a,b]}.
\end{equation*}
This contradicts the claim of Theorem \ref{Erdos} and concludes the proof.\qed   
\end{proof}

We are now, in a position to give a generalization of Theorem \ref{BoundTheoremreal} for the parabolic
stability radius with damping, $\ell_{m,n}(\eta,\delta)$.

\begin{theorem}
\label{DampTheorem}
For any positive integers $m,n$ such that $m \geq n$, we have $\ell_{m,n}(\eta,\delta) \leq -\xi$, where 
$\xi$ is the smallest negative solution of the polynomial equation  
\begin{equation*}
(x + \delta)^{m-n}\mathcal{L}^{(-m-1)}_{n}(x) - (1- \eta)\binom{m}{n}
\sum_{i=0}^{m-n} (-1)^{m-n-1} \frac{\binom{2m}{2i} \binom{m-n}{i}}{\binom{m}{i}} x^i 
\delta^{m-n-i} = 0. 
\end{equation*}
\end{theorem}  
\begin{proof}
To allege the notation, we set $\ell_{m,n}:=\ell_{m,n}(\eta,\delta)$. 
Let $\bar{\Phi}_{m,n} \in \Pi_{m,n}$ be the unique polynomial such that
$|\bar{\Phi}_{m,n}(x)| \leq 1 - \eta <1$ 
for any $x \in [-\ell_{m,n},-\delta]$. Thus according to Corollary 
\ref{GeneralLubinsky} with $[a,b] =  [-\ell_{m,n},-\delta]$ 
\begin{equation*}
|\bar{\phi}_{m,n}\left(-\ell_{m,n}^{[n]},0^{[m-n]} \right)| \leq 
|t_{m} \left(-\ell_{m,n}^{[n]},0^{[m-n]} \right)| (1-\eta),   
\end{equation*}  
where $t_m$ (resp. $\bar{\phi}_{m,n}$) the polar form of the polynomial $T_m(.,[-\ell_{m,n},-\delta])$
(resp. $\bar{\Phi}_{m,n}$). From (\ref{explicitT}), we have 
\begin{equation*}
t_{m} \left(-\ell_{m,n}^{[n]},x^{[m-n]} \right) =  \sum_{i=0}^{m-n} (-1)^{m-i} 
\frac{\binom{2m}{2i}}{\binom{m}{i}} B_{i}^{m-n}(x, [-\ell_{m,n},-\delta]), 
\end{equation*}
we obtain
\begin{equation}
\label{Last}
t_{m} \left(-\ell_{m,n}^{[n]},0^{[m-n]} \right) = \sum_{i=0}^{m-n} (-1)^{n} 
\frac{\binom{2m}{2i} \binom{m-n}{i}}{\binom{m}{i}} \frac{\ell_{m,n}^i 
\delta^{m-n-i}}{(\ell_{m,n} - \delta)^{m-n}} 
\end{equation}
We conclude the proof using Corollary \ref{GeneralLubinsky} in 
conjunction with (\ref{Last}) and Proposition \ref{PolarLaguerre}.\qed
\end{proof}
Note that when we set $\delta = \eta =0$ in Theorem \ref{DampTheorem}, 
we recover the statement of Theorem \ref{BoundTheoremreal}.

\section{Concluding remarks and future work}
\label{Sec5}
In this work, the theory of polar forms was instrumental in providing for several bounds
on the stability radii of explicit Runge-Kutta methods. These bounds are achieved by 
using a simple and elegant strategy: Eliminating via the polar form the unknown parameters 
of the problem without destroying the main quantitative feature of the problem. 
To obtain better bounds than the one given in this work, one should incorporate 
into our strategy further properties of the optimal stability polynomial. For instance, it is well known that 
the parabolic optimal polynomials satisfy an alternation property. In particular, the number 
of real zeros of these polynomials is less than their degrees. In principle, such property should lead 
to improved bounds in Lubinsky-Ziegler inequality, which in turn would lead to improved upper bounds 
for the parabolic stability radius. Such program will be carried in a future work. Moreover, 
the techniques used in this work are flexible enough to give bounds on the hyperbolic stability radius.




\end{document}